\documentclass[12pt]{amsart}
\usepackage{latexsym,enumerate}
\usepackage{amssymb}
\usepackage[colorlinks]{hyperref}

\newtheorem{theorem}{Theorem}[section]
\newtheorem{lemma}[theorem]{Lemma}
\newtheorem{proposition}[theorem]{Proposition}
\newtheorem{corollary}[theorem]{Corollary}

\theoremstyle{definition}

\theoremstyle{remark}

\numberwithin{equation}{section}

\newcommand{\GL}{{\mathrm {GL}}}
\newcommand{\PGL}{{\mathrm {PGL}}}
\newcommand{\SL}{{\mathrm {SL}}}
\newcommand{\PSL}{{\mathrm {PSL}}}

\newcommand{\PSU}{{\mathrm {PSU}}}

\newcommand{\GO}{{\mathrm {GO}}}

\newcommand{\SO}{{\mathrm {SO}}}
\newcommand{\PCO}{{\mathrm {PCO}}}

\newcommand{\Sp}{{\mathrm {Sp}}}
\newcommand{\PSp}{{\mathrm {PSp}}}
\newcommand{\PCSp}{{\mathrm {PCSp}}}

\newcommand{\Ker}{\operatorname{Ker}}
\newcommand{\Aut}{{\mathrm {Aut}}}
\newcommand{\Out}{{\mathrm {Out}}}

\newcommand{\Irr}{{\mathrm {Irr}}}
\newcommand{\IBR}{{\mathrm {IBr}}}

\newcommand{\diag}{{\mathrm {diag}}}

\newcommand{\Syl}{{\mathrm {Syl}}}

\newcommand{\Stab}{{\mathrm {Stab}}}

\newcommand{\Cl}{{\mathrm {Cl}}}

\newcommand{\ZZ}{{\mathbb Z}}
\newcommand{\NN}{{\mathbb N}}

\newcommand{\GG}{{\mathbb G}}

\newcommand{\FF}{{\mathbb F}}

\newcommand{\GC}{\mathcal{G}}

\newcommand{\ta}{\hspace{0.5mm}^{2}\hspace*{-0.2mm}}
\newcommand{\tb}{\hspace{0.5mm}^{3}\hspace*{-0.2mm}}

\newcommand{\bC}{{\mathbf{C}}}
\newcommand{\bO}{{\mathbf{O}}}
\newcommand{\bN}{{\mathbf{N}}}

\newcommand{\bZ}{{\mathbf{Z}}}
\newcommand{\Al}{\textup{\textsf{A}}}

\newcommand{\pcore}{\mathbf{O}}
\newcommand{\GammaL}{\operatorname{\Gamma L}}
\newcommand{\IBr}{\operatorname{IBr}}

\begin{document}

\title[Bounding $p$-Brauer characters]
{Bounding $p$-Brauer characters in finite groups with two conjugacy
classes of $p$-elements}

\author[N.\,N. Hung]{Nguyen Ngoc Hung}
\address{Department of Mathematics, The University of Akron, Akron,
OH 44325, USA} \email{hungnguyen@uakron.edu}

\author[B. Sambale]{Benjamin Sambale}
\address{Institut f\"{u}r Algebra, Zahlentheorie und Diskrete Mathematik,
Leibniz Universit\"{a}t Hannover, Welfengarten 1, 30167 Hannover,
Germany} \email{sambale@math.uni-hannover.de}

\author[P.\,H. Tiep]{Pham Huu Tiep}
\address{Department of Mathematics, Rutgers University, Piscataway, NJ 08854, USA}
\email{tiep@math.rutgers.edu}

\thanks{The second author is supported by the German Research
Foundation (\mbox{SA 2864/1-2} and \mbox{SA 2864/3-1}). The third author gratefully acknowledges the support of the NSF (grant
DMS-1840702), the Joshua Barlaz Chair in Mathematics, and the
Charles Simonyi Endowment at the Institute for Advanced Study
(Princeton). }
\dedicatory{Dedicated to Burkhard K\"ulshammer on the occasion of his retirement.}



\begin{abstract} Let $k(B_0)$ and $l(B_0)$ respectively denote the number of
ordinary and $p$-Brauer irreducible characters in the principal
block $B_0$ of a finite group $G$. We prove that, if
$k(B_0)-l(B_0)=1$, then $l(B_0)\geq p-1$ or else $p=11$ and
$l(B_0)=9$. This follows from a more general result that for every
finite group $G$ in which all non-trivial $p$-elements are
conjugate, $l(B_0)\geq p-1$ or else $p = 11$ and $G/\bO_{p'}(G)
\cong C^2_{11} \rtimes\SL(2,5)$. These results are useful in the
study of principal blocks with few characters.

We propose that, in every finite group $G$ of order divisible by $p$,
the number of irreducible Brauer characters in the principal
$p$-block of $G$ is always at least $2\sqrt{p-1}+1-k_p(G)$, where $k_p(G)$
is the number of conjugacy classes of $p$-elements of $G$. This
indeed is a consequence of the celebrated Alperin weight conjecture
and known results on bounding the number of $p$-regular classes in
finite groups.
\end{abstract}

\maketitle


\section{Introduction}

Let $G$ be a finite group and $p$ a prime. Bounding the number
$k(G)$ of conjugacy classes of $G$ and the number $k_{p'}(G)$ of
$p$-regular conjugacy classes of $G$ is a classical problem in group
representation theory, one important reason being that $k(G)$ is the
same as the number of non-similar irreducible complex
representations of $G$ and $k_{p'}(G)$ is the same as the number of
non-similar irreducible representations of $G$ over an algebraically
closed field $\mathbb{F}$ of characteristic $p$. It was shown recently in
\cite[Theorem 1.1]{Hung-Maroti} that if $G$ has order divisible by
$p$, then $k_{p'}(G)\geq 2\sqrt{p-1}+1-k_p(G)$, where $k_{p}(G)$
denotes the number of conjugacy classes of $p$-elements of $G$. As
it is obvious from the bound itself that equality could occur only
when $p-1$ is a perfect square, a ``correct'' bound remains to be
found.

Motivated by the study of blocks which contain a small number of
characters, in this paper we focus on the extremal situation where
$G$ has a unique non-trivial conjugacy class of $p$-elements.

\begin{theorem}\label{theorem-kp'}
Let $p$ be a prime and $G$ a finite group in which all non-trivial
$p$-elements are conjugate. Then one of the following holds:
\begin{enumerate}[\rm(i)]
\item $k_{p'}(G)\ge p$.
\item $k_{p'}(G)=p-1$ and $G\cong C_p\rtimes C_{{p-1}}$ (Frobenius group).
\item $p = 11$, $G \cong C^2_{11} \rtimes\SL(2,5)$ (Frobenius group) and $k_{p'}(G)=9$.
\end{enumerate}
\end{theorem}

Finite groups with a unique non-trivial conjugacy class of
$p$-elements arise naturally from block theory. For a $p$-block $B$
of a group $G$, as usual let $\Irr(B)$ and $\IBR(B)$ respectively
denote the set of irreducible ordinary characters of $G$ associated
to $B$ and the set of irreducible Brauer characters of $G$
associated to $B$, and set $k(B):=|\Irr(B)|$ and $l(B):=|\IBR(B)|$.
The difference $k(B)-l(B)$ is one of the important invariants of the
block $B$ as it somewhat measures the complexity of $B$, and in
fact, the study of blocks with small $k(B)-l(B)$ has attracted
considerable interest, see
\cite{Kulshammer-Navarro-Sambale-Tiep,KoshitanikB3,Rizo} and
references therein.

It is well-known that $k(B)-l(B)=0$ if and only if $k(B)=l(B)=1$, in
which case the defect group of $B$ is trivial. What happens when
$k(B)-l(B)=1$? Brauer's formula for $k(B)$ (see \cite[p.
7]{Kulshammer-Navarro-Sambale-Tiep}) then implies that all
non-trivial $B$-subsections are conjugate. (Recall that a
$B$-subsection is a pair $(u, b_u)$ consisting of a $p$-element
$u\in G$ and a $p$-block $b_u$ of the centralizer $\bC_G(u)$ such
that the induced block $b_u^G$ is exactly $B$.) Therefore, if $B_0$
is the principal $p$-block of $G$ and $k(B_0)-l(B_0)=1$, then all
the non-trivial $p$-elements of $G$ are conjugate.

Given a $p$-block $B$ of $G$,
the well-known blockwise Alperin weight conjecture (BAW) claims that
$l(B)$ is equal to the number of $G$-conjugacy classes of
$p$-weights of $B$ (for details see Section~3). The conjecture
implies $l(B_0)\ge l(b_0)$ where $B_0$ and $b_0$ are the principal
blocks of $G$ and $\bN_G(P)$ respectively. It is easy to see that
$l(b_0)=k_{p'}(\bN_G(P)/\mathbf{O}_{p'}(\bN_G(P)))$.

Now suppose that $k_p(G)=2$. Then the main result of
\cite{Kulshammer-Navarro-Sambale-Tiep} asserts that, aside from very
few exceptions, the Sylow $p$-subgroups of $G$ are (elementary)
abelian, and so let us assume for a moment that $P\in\Syl_p(G)$ is
abelian. It follows that $\bN_G(P)$ controls $G$-fusion in $P$, and
thus $\bN_G(P)/\mathbf{O}_{p'}(\bN_G(P))$ has a unique non-trivial
conjugacy class of $p$-elements as $G$ does. Therefore, the BAW
conjecture and (the $p$-solvable case of) Theorem~\ref{theorem-kp'}
suggest the following, which we are able to prove using only the
known cyclic Sylow case of the conjecture.

\begin{theorem}\label{main-l(Bo)>p-1}
Let $p$ be a prime and $G$ a finite group in which all non-trivial
$p$-elements are conjugate. Let $B_0$ denote the principal $p$-block
of $G$. Then one of the following holds:
\begin{enumerate}[\rm(i)]
\item $l(B_0)\ge p$.
\item $l(B_0)=p-1$ and $\bN_G(P)/\mathbf{O}_{p'}(\bN_G(P))\cong C_p\rtimes C_{{p-1}}$ (Frobenius group).
\item $p = 11$, $G/\mathbf{O}_{p'}(G)
\cong C^2_{11} \rtimes\SL(2,5)$ (Frobenius group) and $l(B_0)=9$.
\end{enumerate}
\end{theorem}

Theorem \ref{main-l(Bo)>p-1} implies that if $G$ is a finite group
with $k_p(G)=2$ then $k(B_0)\geq p$ or $p = 11$ and $k(B_0)=10$.
Indeed, we obtain the following. Here, $k_{0}(B)$ denotes the number
of irreducible ordinary characters of height $0$ in $B$.

\begin{theorem}\label{main-ko>p}
Let $p$ be a prime and $G$ a finite group in which all non-trivial
$p$-elements are conjugate. Let $B_0$ denote the principal $p$-block
of $G$. Then $k_{0}(B_0)\geq p$ or $p = 11$ and $k_0(B_0)=10$.
\end{theorem}

We mention another consequence, which is useful in the study of
principal blocks with few characters, in particular the case
$k(B_0)-l(B_0)=1$. Note that by \cite[Theorem
3.6]{Kulshammer-Navarro-Sambale-Tiep}, the Sylow $p$-subgroups of
$G$ then must be (elementary) abelian, and hence by
\cite{Kessar-Malle}, $k_0(B_0)=k(B_0)$.

\begin{corollary}\label{cor-main}
Let $p$ be a prime and $G$ a finite group with principal $p$-block $B_0$. If $k(B_0)-l(B_0)=1$ then $k_0(B_0)=
k(B_0)\ge p$ or $p=11$ and $k(B_0)=10$.
\end{corollary}

For a quick example, let us assume that $k(B_0)=4$ and $l(B_0)=3$. Then
Corollary~\ref{cor-main} implies that $p\leq 4$, and since the case
$p=3$ is eliminated by \cite[Corollary 1.6]{Landrock81}, one ends up
with $p=2$, implying that the defect group of $B_0$ must be of order
4 by \cite[Corollary 1.3]{Landrock81}, and thus is the Klein four
group. This result was recently proved in \cite[\S5]{KoshitanikB3}.
(See Section~\ref{section-applications} for more examples with
$k(B_0)=l(B_0)+1=5$ and $k(B_0)=l(B_0)+1=7$.)

In Section~\ref{sec:kp(G)=3} we go one step further and prove that
$k_{p'}(G)\ge (p-1)/2$ for finite groups $G$ with at most three
classes of $p$-elements. As explained in
Section~\ref{section-Alperin}, this and the BAW conjecture then
imply that $l(B_0)\geq (p-1)/2$ for principal blocks $B_0$ of groups
with $1<k_p(G)\leq 3$. In general, we propose that $l(B_0)\geq
2\sqrt{p-1}+1-k_p(G)$ for arbitrary groups of order divisible by
$p$, and this follows from \cite[Theorem 1.1]{Hung-Maroti} and again
the BAW conjecture. We should mention that our proposed bound
complements the conjectural upper bound for the number $l(B)$
proposed by Malle and Robinson \cite{malle-robinson}, namely
$l(B)\leq p^{r(B)}$, where $r(B)$ is the \emph{sectional $p$-rank}
of a defect group of $B$.

The paper is organized as follows. In the next
Section~\ref{sec:p-solvable}, we prove Theorem~\ref{theorem-kp'} for
$p$-solvable groups. In Section~\ref{section-Alperin} we make a
connection between Theorem~\ref{main-l(Bo)>p-1} and other bounds on
$l(B_0)$ with the BAW conjecture. Section~\ref{sec:reduction}
reduces Theorem~\ref{main-l(Bo)>p-1} to almost simple groups of Lie
type, which are then solved in
Section~\ref{sec:B0-almost-simple-gps}. In Section~\ref{sec:kp(G)=3}
we prove a general bound for the number of $p$-regular conjugacy
classes in almost simple groups without any assumption on the number
of $p$-classes, and this will be used to achieve a right bound for
$k_{p'}(G)$ for finite groups $G$ with at most three classes of
$p$-elements. Finally, the proof of Theorem~\ref{main-ko>p} and more
examples of applications of Theorem~\ref{main-l(Bo)>p-1} are
presented in Section~\ref{section-applications}.


\section{\texorpdfstring{$p$}{p}-Solvable
groups}\label{sec:p-solvable}

We begin by proving Theorem \ref{theorem-kp'} for $p$-solvable
groups.

\begin{theorem}\label{psolv}
Let $G$ be a $p$-solvable group with $k_p(G)=2$.
Then one of the following holds:
\begin{enumerate}[\rm(i)]
\item $k_{p'}(G)\ge p$.
\item $k_{p'}(G)=p-1$ and $G\cong C_p\rtimes C_{{p-1}}$ (Frobenius group).
\item $p = 11$, $G \cong C^2_{11} \rtimes\SL(2,5)$ (Frobenius group) and $k_{p'}(G)=9$.
\end{enumerate}
\end{theorem}

\begin{proof}
We assume first that $\pcore_{p'}(G)=1$.
Then $P:=\pcore_p(G)\ne 1$ by the Hall--Higman lemma. Since every $p$-element is conjugate to an element of
$P$, $P$ must be a Sylow $p$-subgroup. Since $\bZ(P)\unlhd G$, it
follows that $P=\bZ(P)$ is elementary abelian. Moreover,
$\bC_G(P)=P$ and $G/P$ is a transitive linear group (on $P$). We
need to show that $k_{p'}(G)=k_{p'}(G/P)=k(G/P)\ge p-1$ excluding
the exceptional case. By Passman's classification~\cite{Passman},
$G/P$ is a subgroup of the semilinear group
\[\GammaL(1,p^n)\cong\FF_{p^n}^\times\rtimes\Aut(\FF_{p^n})\cong C_{p^n-1}\rtimes C_n\]
where $P\cong \FF_{p^n}$ or one of finitely many exceptions. We start with the first
case.
Since $\Aut(\FF_{p^n})$ fixes some $x\in P\setminus\{1\}$ in
the base field $\FF_p$, $\FF_{p^n}^\times$ must be contained in
$G/P$ (otherwise $G/P$ cannot be transitive on $P\setminus\{1\}$). Now
$G/P$ has at least $(p^n-1)/n\ge p-1$ conjugacy classes lying inside
$\FF_{p^n}^\times$. The equality here occurs if and only if $n=1$,
in which case $G$ is the Frobenius group $C_p\rtimes C_{p-1}$.

Now suppose that $G/P$ is one of the exceptions in Passman's list
(see \cite[Theorem~15.1]{habil} for detailed information). For $p=
3$ the claim reduces to $|G/P|\ge 3$ which is obviously true. The
remaining cases can be checked by computer. It turns out that
$G\cong C_{11}^2\rtimes\SL(2,5)$ with $p=11$ is the only exception.

Finally, suppose that $N:=\pcore_{p'}(G)\ne 1$. Since $k_p(G)=k_p(G/N)$, the above arguments apply to $G/N$.
Since at least one $p$-regular element lies in $N\setminus\{1\}$, we obtain
\[k_{p'}(G)\ge 1+k_{p'}(G/N)\ge p\]
unless $p=11$ and $G/N\cong C_{11}^2\rtimes\SL(2,5)$. Suppose in this case that $k_{11'}(G)=10$. Then all non-trivial elements of $N$ are conjugate in $G$. As before, $N$ must be an elementary abelian $q$-group for some prime $q\ne 11$.
Let $N\le M\unlhd G$ such that $M/N\cong C_{11}^2$. Then $G/M$ acts transitively on the $M$-orbits of $N\setminus\{1\}$. In particular, these $M$-orbits have the same size. Since the non-cyclic group $M/N$ cannot act fixed point freely on $N$, all $M$-orbits have size $1$ or $11$. In the second case, $(|N|-1)/11$ divides $|G/M|=120$. This leaves only the possibility that $N$ is cyclic of order $q\ge 23$. But then $G/\bC_G(N)$ is cyclic and we derive the contradiction $G=G'N\le \bC_G(N)$.

It remains to deal with the case where $M$ acts trivially on $N$. Here we may go over to $\overline{G}:=G/\pcore_{11}(G)$ such that $k(\overline{G})=k_{11'}(G)=10$. Since $\overline{G}$ acts transitively on $\overline{N}\setminus\{1\}$, we obtain that $|\overline{N}|-1$ divides $|\overline{G}/\overline{N}|=|G/M|=120$.
Since $\overline{G}/\bC_{\overline{G}}(\overline{N})\in\{\SL(2,5),A_5\}$, this leaves the possibilities $|\overline{N}|\in\{2^4,5^2\}$.
Now it can be checked by computer that there is no (perfect) group with these properties.
\end{proof}

Apart from finitely many exceptions, the proof actually shows that
$k_{p'}(G)\ge\frac{p^n-1}{n}$ where $|G|_p=p^n$.

The following result provides a bound for $k_{p'}(G)$ in
$p$-solvable groups with three conjugacy classes of $p$-elements.

\begin{theorem}\label{psolv2}
Let $G$ be a $p$-solvable group with $k_p(G)=3$. Then $k_{p'}(G)\ge
(p-1)/2$ with equality if and only if $p>2$ and $G$ is the Frobenius
group $C_p\rtimes C_{(p-1)/2}$.
\end{theorem}

\begin{proof}
As in the proof of Theorem~\ref{psolv} we start by assuming $\pcore_{p'}(G)=1$.
Since the claim is easy to show for $p\le 5$, we may assume that $p\ge 7$ in the following.

Let $P:=\pcore_p(G)\ne 1$ and $H:=G/P$. Suppose first that $|H|$ is divisible by $p$. Then
$k_p(H)=2$ and
\[k_{p'}(G)=k_{p'}(H)\ge p-2>\frac{p-1}{2}\]
by Theorem~\ref{psolv}. Now let $H$ be a $p'$-group. Suppose that $P$
possesses a characteristic subgroup $1<Q<P$. Then $P\setminus Q$ must
be an $H$-orbit and therefore $|P\setminus Q|$ is not divisible by
$p$. This is clearly impossible. Hence, $P$ is elementary abelian
and $G\cong P\rtimes H$ is an affine primitive permutation group of
rank $3$ (i.\,e. a point stabilizer has three orbits on $P$). These
groups were classified by Liebeck~\cite{Liebeckrk3}.

Let $|P|=p^n$.
Suppose first that $H\le\GammaL(1,p^n)$. Then $H$ contains a
semiregular normal subgroup $C\le\FF_{p^n}^\times$. Clearly, $C$ has
exactly $\frac{p^n-1}{|C|}$ orbits on $P\setminus\{1\}$ each of
length $|C|$. Moreover, $\Aut(\FF_{p^n})$ fixes one of these orbits
and can merge at most $n$ of the remaining. Hence, $|C|+n|C|\ge
p^n-1$ and $|C|\ge\frac{p^n-1}{1+n}$. Now there are at least
$\frac{p^n-1}{n+n^2}$ conjugacy classes of $H$ lying
inside $C$. Since
\[\frac{p^n-1}{p-1}\ge 1+p+\ldots+p^{n-1}\ge 1+2+\ldots+2^{n-1}=2^n-1\ge\frac{n(n+1)}{2},\]
we obtain $k_{p'}(H)\ge\frac{p-1}{2}$ with equality if only if $n=1$ and $G\cong
C_p\rtimes C_{(p-1)/2}$.

Now assume that $H$ acts imprimitively on $P=P_1\times P_2$
interchanging $P_1$ and $P_2$. Then $K:=\bN_H(P_1)=\bN_H(P_2)\unlhd
H$ and $K/\bC_H(P_1)$ is a transitive linear group on $P_1$.
Theorem~\ref{psolv} yields $k(K)\ge k(K/\bC_H(P_1))\ge p-2$. Since
$|H:K|=2$, the conjugacy classes of $K$ can only fuse in pairs in
$H$. This leaves at least $1+\frac{p-3}{2}=\frac{p-1}{2}$ conjugacy
classes of $H$ inside $K$ and there is at least one more class
outside $K$. Altogether, $k(H)\ge\frac{p+1}{2}$.

Next suppose that $P=P_1\otimes P_2$ considered as $\FF_q$-spaces
where $q^a=p^n$ and $H$ stabilizes $P_1$ and $P_2$. Here $|P_1|=q^2$
and $|P_2|=q^d\ge q^2$. By \cite[Lemma~1.1]{Liebeckrk3}, $H$ has an
orbit of length $(q^d-1)(q^d-q)$, but this is impossible since $H$
is a $p'$-group.

The cases (A4)--(A11) in Liebeck~\cite{Liebeckrk3} are not
$p$-solvable. Cases (B) and (C) are finitely many exception. Suppose
that $p=7$ and $k(H)\le 3$. It is well-known that then $H\le S_3$
and therefore $|P|\le 1+6+6$. It follows that $n=1$ and $G\cong
C_7\rtimes C_3$. Hence, let $p\ge 11$. From \cite{Liebeckrk3} we
obtain $|P|\le 89^2$. Since the primitive permutation groups of
degree at most $2^{12}-1$ are available in GAP~\cite{GAP48}, we may
assume that $p\ge 67$. There are only three cases left, namely
$p\in\{71,79,89\}$ and $n=2$. Here $A_5\le H/\bZ(H)$. Since $A_5$ is
a maximal subgroup of $\PSL(2,p)$ (see
\cite[Hauptsatz~II.8.27]{Huppert}), it follows that
$H\cap\SL(2,p)=\SL(2,5)$. Consequently,
\[C:=H/\SL(2,5)\le\GL(2,p)/\SL(2,p)\cong C_{p-1}.\]
Since $H$ has an orbit of length at least $(p^2-1)/2$, we obtain
$120|C|=|H|\ge(p^2-1)/2$. This yields $k(H)\ge1+|C|>(p-1)/2$ unless
$p=79$ and $|C|=26$. In this exception, $H=\SL(2,5).2\times C_{13}$
and obviously $k(H)\ge 3\cdot 13=(p-1)/2$.

Finally, suppose that $N:=\pcore_{p'}(G)\ne 1$. Then the above arguments apply to $G/N$ and we obtain
\[k_{p'}(G)\ge 1+k_{p'}(G/N)>\frac{p-1}{2}\]
since at least one non-trivial $p$-regular element lies in $N$.
\end{proof}

We remark that the $p$-solvability assumption in
Theorem~\ref{psolv2} will be removed in Section~\ref{sec:kp(G)=3}.


\section{The blockwise Alperin weight conjecture}\label{section-Alperin}

In this section, we will explain that, when the Sylow $p$-subgroups
of $G$ are cyclic, the main result Theorem~\ref{main-l(Bo)>p-1} (and
also Theorem~\ref{main-kp(G)=3}) is a consequence of the known
cyclic Sylow case of the blockwise Alperin weight (BAW) conjecture
and the $p$-solvable results proved in the previous section.

Let $B$ be a $p$-block of $G$. Recall that $l(B)$ denotes the number
of irreducible Brauer characters of $B$. A $p$-weight for $B$ is a
pair $(Q,\lambda)$ of a $p$-subgroup $Q$ of $G$ and an irreducible
$p$-defect zero character $\lambda$ of $\bN_G(Q)/Q$ such that the
lift of $\lambda$ to $\bN_G(Q)$ belongs to a block which induces the
block $B$. The BAW conjecture claims that $l(B)$ is equal to the
number of $G$-conjugacy classes of $p$-weights of $B$. In
particular, the conjecture implies that $l(B)\geq l(b)$, where $b$
is the Brauer correspondent of $B$ (see
\cite[Consequence~1]{Alperin}). In fact, when a defect group of $B$
is abelian, the conjecture is equivalent to $l(B)=l(b)$ (see
\cite[Consequence~2]{Alperin}).

Let $P\in\Syl_p(G)$, and let $B_0$ and $b_0$ be respectively the
principal blocks of $G$ and $\bN_G(P)$. Assume that the BAW
conjecture holds for $(G,p)$.
Since $\bN_G(P)$ is $p$-solvable, \cite[Theorems~9.9 and 10.20]{Navarro}
show that
\[
l(B_0)\geq l(b_0)=k_{p'}(\bN_G(P)/\mathbf{O}_{p'}(\bN_G(P)))=k(\bN_G(P)/P\bC_G(P)).
\]

By Burnside's fusion argument (see \cite[Lemma 5.12]{Isaac}),
$H:=\bN_G(P)/P\bC_G(P)$ controls fusion in $Z:=\bZ(P)$. In
particular, $k_p(Z\rtimes H)\le k_p(G)$.

Combining the above analysis with the results of the previous
section, we deduce that, if $k_p(G)=2$ then $k_p(Z\rtimes H)=2$ and
$l(B_0)\geq p-1$ or $p=11$ and
$\bN_G(P)/\mathbf{O}_{p'}(\bN_G(P))\cong C^2_{11} \rtimes\SL(2,5)$.
Similarly, if $k_p(G)=3$ then $l(B_0)\geq (p-1)/2$, and thus
$k_{p'}(G)\geq (p-1)/2$. Also, when $k_p(G)=2$, $l(B_0)=p-1$ if and
only if $k_{p'}(\bN_G(P)/\mathbf{O}_{p'}(\bN_G(P)))=p-1$, which
occurs if and only if $\bN_G(P)/\mathbf{O}_{p'}(\bN_G(P))$ is
isomorphic to the Frobenius group $C_p\rtimes C_{{p-1}}$, by
Theorem~\ref{psolv}.

We therefore have the following, which was already mentioned in the
introduction.

\begin{proposition}
Let $p$ be a prime and $G$ a finite group with $k_p(G)=3$. Let $B_0$
be the principal block of $G$. Then the blockwise Alperin weight
Conjecture (for $B_0$) implies that $l(B_0)\geq (p-1)/2$.
\end{proposition}

\begin{proposition}
Let $p$ be a prime and $G$ a finite group of order divisible by $p$.
Let $B_0$ be the principal block of $G$. Then the blockwise Alperin
weight Conjecture (for $B_0$) implies that $l(B_0)\geq 2\sqrt{p-1}+1-k_p(G)$.
\end{proposition}

\begin{proof}
This follows from the above analysis and \cite[Theorem
1.1]{Hung-Maroti}.
\end{proof}

We have seen that Theorem \ref{main-l(Bo)>p-1} holds for $(G,p)$ if
the BAW conjecture holds for $(G,p)$. In particular, by Dade's
results \cite{Dade66} on blocks with cyclic defect groups, we have
proved the main results for groups with cyclic Sylow $p$-subgroups.

We end this section by another consequence of the BAW conjecture on
possible values of $k(B)$ and $l(B)$ in blocks with $k(B)-l(B)=1$.
In the following theorem we make use of Jordan's totient function
$J_2:\NN\to\NN$ defined by
\[J_2(n):=n^2\prod_{p\mid n}\frac{p^2-1}{p^2}\]
where $p$ runs through the prime divisors of $n$ (compare with the definition of Euler's function $\phi$).

\begin{theorem}\label{thm:k-l=1}
Let $B$ be a $p$-block of a finite group $G$ with defect $d$ such
that $k(B)-l(B)=1$. Suppose that $B$ satisfies the Alperin weight
Conjecture. Then one of the following holds:

\begin{enumerate}[\rm(i)]
\item\label{dickson}
$d=nk$ such that all prime divisors of $n$ divide $p^k-1$. Moreover,
if $4$ divides $n$, then $4$ divides $p^k-1$. Here
\[l(B)=\sum_{e\mid n}\frac{p^{ek}-1}{ne}J_2(n/e).\]
In particular, $l(B)=p^d-1$ if $n=1$ and $l(B)>(p^k-1)\phi(n)+\frac{p^d-1}{n^2}$ if $n>1$.

\item\label{excep}
\[\begin{array}{c|ccccccc}
p^d&5^2&7^2&11^2&11^2&23^2&29^2&59^2\\\hline
l(B)&7&8&9&35&88&63&261
\end{array}\]
\end{enumerate}
Conversely, all values for $l(B)$ given in (i) and (ii) do occur in
examples.
\end{theorem}

\begin{proof}
By \cite[Theorem~7.1]{HKKS}, $B$ has an elementary abelian defect group $D$. The equation $k(B)-l(B)=1$ implies further that the inertial quotient $E$ of $B$ acts regularly on $D\setminus\{1\}$. It follows that all Sylow subgroups of $E$ are cyclic or quaternion groups. In particular, $E$ has trivial Schur multiplier. Hence, the Alperin weight conjecture asserts that $l(B)=k(E)$ (see \cite[Conjecture 2.6]{habil} for instance). Note that $D\rtimes E$ is a sharply $2$-transitive group on $D$ and those were classified by Zassenhaus~\cite{Zassenhaus1} (see also \cite[Section~7.6]{DM}). Apart from the seven exceptions described in \eqref{excep},
$D\rtimes E$ arises from a Dickson near-field $F$ where $(F,+)\cong D$ and $F^\times\cong E$. More precisely, there exists a factorization $d=nk$ as in \eqref{dickson} such that $F$ can be identified with $\FF_{q^n}$ where $q=p^k$ and the multiplication is modified as follows. Let $\FF_{q^n}^\times=\langle\zeta\rangle$. Let $\gamma:\FF_{q^n}\to\FF_{q^n}$, $x\mapsto x^q$ be the Frobenius automorphism of $\FF_{q^n}$ with respect to $\FF_q$. The hypotheses imply (with some effort) that $q$ has multiplicative order $n$ modulo $(q-1)n$. Hence, for every integer $a$ there exists a unique integer $a^*$ such that $0\le a^*<n$ and
\[q^{a^*}\equiv 1+a(q-1)\pmod{(q-1)n}.\]
It is easy to check that $\Gamma:\FF_{q^n}^\times\to\langle\gamma\rangle$, $\zeta^a\mapsto \gamma^{a^*}$ is an epimorphism.
We define
\[F^\times:=\bigl\{(\zeta^a,\gamma^{a^*}):0\le a<q^n-1\bigr\}\le\FF_{q^n}^\times\rtimes\langle\gamma\rangle=\GammaL(1,q^n).\]
Note that $F^\times$ is just the Singer cycle $\FF_q^\times$ if $n=1$. Although different choices for $\zeta$ may lead to non-isomorphic near-fields, the group $F^\times$ is certainly uniquely defined (as a subgroup of $(\ZZ/(q^n-1)\ZZ)\rtimes(\ZZ/n\ZZ)$ for instance).

It is easy to check that $A:=\langle(\zeta^n,1)\rangle=\Ker\Gamma\unlhd F^\times$ and $F^\times/A\cong C_n$. This makes it possible to compute $k(E)=k(F^\times)$ via Clifford theory with respect to $A$. The natural actions of $F^\times$ on $A$ and on $\Irr(A)$ are permutation isomorphic, by Brauer's permutation lemma. Thus, instead of counting characters of $A$ with a specific order we may just count elements. For a divisor $e\mid n$, let $\alpha(e)$ be the number of elements in $F^\times\cap\FF_{q^e}$ which do not lie in any proper subfield of $\FF_{q^e}$. Then
\[\beta(e):=|F^\times\cap\FF_{q^e}|=\frac{q^e-1}{e}=\sum_{f\mid e}\alpha(f).\]
By M\"obius inversion we obtain
\[\alpha(e)=\sum_{f\mid e}\mu(e/f)\frac{q^f-1}{f}.\]
This is also the number of characters in $\Irr(A)$ with inertial index $e$. These characters distribute into $\alpha(e)/e$ orbits under $F^\times$. Each such character has $n/e$ distinct extensions to its inertial group and each such extension induces to an irreducible character of $F^\times$. The number of character of $F^\times$ obtained in this way is therefore $\alpha(e)n/e^2$. In total,
\[l(B)=k(E)=k(F^\times)=\sum_{e\mid n}\frac{n}{e^2}\sum_{f\mid e}\mu(e/f)\frac{q^f-1}{f}.\]
Now observe that $n^2=\sum_{d\mid n}J_2(d)$ for all $n\ge 1$. Hence, another M\"obius inversion yields
\[\sum_{e\mid n}\frac{n}{e^2}\sum_{f\mid e}\mu(e/f)\frac{q^f-1}{f}=\sum_{f\mid n}\frac{q^f-1}{fn}\sum_{e'\mid\frac{n}{f}}\Bigl(\frac{n}{e'f}\Bigr)^2\mu(e')=\sum_{f\mid n}\frac{q^f-1}{fn}J_2(n/f).\]
If $n>1$, then $n\phi(n)=n^2\prod_{p\mid n}\frac{p-1}{p}<J_2(n)$ and the second claim follows.

Conversely, if $d=nk$ satisfies the condition in \eqref{dickson}, then a corresponding near-field $F$ can be constructed as above. This in turn leads to a sharply $2$-transitive group $G=F\rtimes F^\times$. Now $G$ has only one block $B$, namely the principal block, and $l(B)=k(F^\times)$ is given as above.
\end{proof}


\section{Reduction for Theorem \ref{main-l(Bo)>p-1}}\label{sec:reduction}

In this section we prove Theorem~\ref{main-l(Bo)>p-1}, assuming a
result on bounding $l(B_0)$ for almost simple groups of Lie type
that will be proved in Section~\ref{sec:B0-almost-simple-gps}. We restate Theorem~\ref{main-l(Bo)>p-1} for the convenience of the reader.

\begin{theorem}\label{main-l(Bo)>p-1repeat}
Let $p$ be a prime and let $G$ be a finite group with $k_p(G)=2$.
Let $B_0$ be the principal $p$-block of $G$. Then $l(B_0)\geq p-1$ or $p=11$ and
$G/\pcore_{p'}(G)\cong C_{11}^2\rtimes\SL(2,5)$. Furthermore,
$l(B_0)=p-1$ if and only if $\bN_G(P)/\mathbf{O}_{p'}(\bN_G(P))$ is
isomorphic to the Frobenius group $C_p\rtimes C_{{p-1}}$.
\end{theorem}

\begin{proof}
Recall that $B_0$ is isomorphic to the principal $p$-block of
$G/\pcore_{p'}(G)$. We may assume that $\pcore_{p'}(G)=1$. Moreover, as the theorem is easy for
$p=2$, we assume $p\ge 3$. Also, since the case of cyclic Sylow
follows from the blockwise Alperin weight conjecture, as explained
in Section~\ref{section-Alperin}, we assume furthermore that
$P\in\Syl_p(G)$ is not cyclic. We aim to prove that $l(B_0)> p-1$ or
$p=11$ and $G\cong C_{11}^2\rtimes\SL(2,5)$.

Assume first that $P$ is non-abelian. Then $p\leq 5$ by the main
result of \cite{Kulshammer-Navarro-Sambale-Tiep}. When $p=5$, $G$ is
isomorphic to the sporadic simple Thompson group $Th$, and from the
Atlas~\cite{Atl1} we get $l(B_0)=l(B_0(Th))=20> 4$, as desired. Let $p=3$. Then
$S:=\bO^{p'}(G)$ is isomorphic to the Rudvalis group $Ru$, the Janko
group $J_4$, the Tits group $\ta F_4(2)'$, or the Ree groups $\ta
F_4(q)$ with $q = 2^{6b\pm1}$ for $b\in\ZZ^{+}$, by
\cite{Kulshammer-Navarro-Sambale-Tiep} again.
Since $\bC_G(S)\le\pcore_{p'}(G)=1$, $G$ is almost simple.
We now check with \cite{GAP48} that
\[l(B_0(Ru))=l(B_0(J_4))=l(B_0(\ta F_4(2)'))=l(B_0(\ta F_4(2)))=9>2.\]
Therefore we may assume that $S=\ta F_4(q)$ with $q = 2^{6b\pm1}$
for some $b\in\ZZ^+$ and $S\unlhd G \leq \Aut(S)$. By \cite[\S 6 and
\S 7]{Malle90} (see also \cite[Table C5]{Himstedt11}), the principal
3-block of $\ta F_4(q)$ ($q\geq 8$) contains three irreducible
Brauer characters (denoted by $\phi_{21}$, $\phi_{5,1}$, and of
course the trivial character) that are $\Aut(S)$-invariant (since
their degrees are unique in $B_0(S)$), and thus we have $l(B_0)\geq
3$, as wanted.

We may now assume that $P$ is abelian. By Burnside's fusion
argument, all non-trivial $p$-elements of $\bN_G(P)$ are conjugate,
i.\,e. $\bN_G(P)$ satisfies the hypothesis of Theorem~\ref{psolv}. Let
$N$ be a minimal normal subgroup of $G$. If $N$ is elementary
abelian, then $N=P$ since every element of $P$ is conjugate to some
element of $N$. From $\pcore_{p'}(G)=1$ it then follows that $B_0$
is the only block of $G$. Hence, the theorem follows from
Theorem~\ref{psolv}. Now let $N=T_1\times\ldots\times T_n$ with
non-abelian simple groups $T_1\cong\ldots\cong T_n$. Since
$\pcore_{p'}(G)=1$, $|T_i|$ is divisible by $p$. Since non-trivial
$p$-elements of the form $(x,1,\ldots,1)$ and $(x,x,1,\ldots,1)$ in
$N$ cannot be conjugate in $G$, we conclude that $n=1$, i.\,e. $N$
is simple. Since $\bC_G(N)\cap N=\bZ(N)=1$ we have
$\bC_G(N)\le\pcore_{p'}(G)=1$. Altogether, $G\le\Aut(N)$, i.\,e. $G$
is an almost simple group. Moreover, $p\nmid |G/N|$.

Let $N=\Al_n$ be an alternating group. Recall that the Sylow $p$-subgroups
of $G$ (and $N$) are not cyclic. Therefore, $n\geq 2p$. But
then the $p$-elements of cycle type $(p)$ and $(p,p)$ are not
conjugate in $G$. The sporadic and the Tits groups can be checked
with \cite{GAP48} (or one appeals to Alperin's weight conjecture proved in
\cite{SambaleBroue}). Next let $N$ be a simple group of Lie type in
characteristic $p$. Then $P$ can only be abelian if
$N\cong\PSL(2,p^n)$ for some $n\ge 1$ (see \cite[Proposition
5.1]{Sawabe-Watanabe} for instance). In this case, Alperin's weight
conjecture is known to hold for $B_0$, i.\,e. $l(B_0)=l(b_0)$ where
$b_0$ is the principal block of $\bN_G(P)$. Now $l(b_0)$ is the
number of $p$-regular conjugacy classes of the $p$-solvable group
$H:=\bN_G(P)/\pcore_{p'}(\bN_G(P))$. Hence, the claim follows from
Theorem~\ref{psolv} unless possibly $p=11$ and
$H=C_{11}^2\rtimes\SL(2,5)$. Then however $N\cong\PSL(2,11^2)$ and
$\SL(2,5)$ is not involved in $\bN_G(P)$.

Finally, let $N$ be a simple group of Lie type in characteristic
different from $p$. In such case, we show in
Theorem~\ref{thm:B0-Lie-type} below that $l(B_0)\ge p$, and thus the
proof is complete.
\end{proof}


\section{Principal blocks of almost simple groups of Lie type}\label{sec:B0-almost-simple-gps}

We now prove the following result which is left off at the end of
Section~\ref{sec:reduction}.

\begin{theorem}\label{thm:B0-Lie-type}
Let $p\geq 3$ be a prime and $S\neq \ta F_4(2)'$ a simple group of
Lie type in characteristic different from $p$. Assume that the Sylow
$p$-subgroups of $S$ are abelian but not cyclic. Let $S\unlhd G \leq
\Aut(S)$ such that $p\nmid |G/S|$. Let $B_0$ be the principal
$p$-block of $G$. Then $l(B_0)\ge p$ or $G$ has at least two classes
of non-trivial $p$-elements.
\end{theorem}

We will work with the following setup. Let $\GC$ be a simple
algebraic group of simply connected type defined over $\FF_q$ and
$F$ a Frobenius endomorphism on $\GC$ such that and
$S=\GG/\bZ(\GG)$, where $\GG:=\GC^{F}$ is the set of fixed points of
$\mathcal{G}$ under $F$. Let $\GC^\ast$ be an algebraic group with a
Frobenius endomorphism which, for simplicity, we denote by the same
$F$, such that $(\GC,F)$ is in duality to $(\GC^\ast,F)$. Set
$\GG^\ast:={\GC^\ast}^F$. As we will see below, the Brauer
characters in the principal blocks of $S$ and $\GG$ arise from the
so-called unipotent characters of $\GG$. These are the irreducible
characters of $\GG$ occurring in a Deligne--Lusztig character
$R^\GC_\mathcal{T}(1)$, where $\mathcal{T}$ runs over the $F$-stable
maximal tori of $\GC$, see \cite[Definition 13.19]{Digne-Michel91}.
It is well-known that the unipotent characters of $\GG$ all have
$\bZ(\GG)$ in their kernel, and so they are viewed as (unipotent)
characters of $S$.

From the assumption on $P\in\Syl_p(S)$ and $p$, we may assume that
$S$ is not one of the types $A_1$, $\ta G_2$, and $\ta B_2$. Assume
for a moment that $S$ is also not a Ree group of type $\ta F_4$
neither, so that $F$ defines an $\FF_q$-rational structure on $\GC$.
Let $d$ be the multiplicative order of $q$ modulo $p$.

By \cite[Theorem A]{kessar-Malle2}, which includes earlier results
of Brou\'{e}--Malle--Michel \cite{Broue-Malle-Michel} and of
Cabanes--Enguehard \cite{Cabanes-Enguehard94}, the $p$-blocks of
$\GG$ are parameterized by $d$-cuspidal pairs
$(\mathcal{L},\lambda)$ of a $d$-split Levi subgroup $\mathcal{L}$
of $\GC$ and a $d$-cuspidal unipotent character $\lambda$ of
$\mathcal{L}^F$. In particular, the principal block of $\GG$
corresponds to the pair consisting of the centralizer
$\mathcal{L}_d:=\bC_{\GC}(\mathcal{S}_d)$ of a Sylow $d$-torus
$\mathcal{S}_d$ of $\GC$ and the trivial character of
$\mathcal{L}_d^F$. Moreover, the number of unipotent characters in
$B_0(\GG)$ is the same as the number of characters in the
$d$-Harish-Chandra series associated to the pair
$(\mathcal{L}_d,1)$. By \cite[Theorem 3.2]{Broue-Malle-Michel},
characters in each $d$-Harish-Chandra series are in one-to-one
correspondence with the irreducible characters of the relative Weyl
group of the $d$-cuspidal pair defining the series. Therefore, the
number of unipotent characters in $B_0(\GG)$ is precisely the number
of irreducible characters of the relative Weyl group
$W(\mathcal{L}_d)$ of $\mathcal{L}_d$.

Assume that $p\nmid |\bZ(\GG)|$. Then, as the Sylow $p$-subgroups of
$S$ are abelian, those of $\GG$ are abelian as well. In such
situation, we follow \cite[\S 5.3]{Malle-Maroti} to control the
number of conjugacy classes of $p$-elements in $\GG$. In particular,
by \cite[Proposition 2.2]{Malle14}, we know that the order $d$ of
$q$ modulo $p$ defined above is a unique positive integer such that
$p\mid \Phi_d(q)$ with $\Phi_d$ the $d$th cyclotomic polynomial
dividing the generic order of $\GG$. Furthermore, $p$ is indeed a
good prime for $\GC$ (see \cite[Lemma 2.1]{Malle14}). Let
$\Phi_d^{m_d}$ be the precise power of $\Phi_d$ dividing the generic
order of $\GG$. Note that, by the assumption on Sylow $p$-subgroups
of $S$ and the main result of
\cite{Kulshammer-Navarro-Sambale-Tiep}, a $P\in\Syl_p(S)$ must be
elementary abelian, and thus $\Phi_d(q)$ is divisible by $p$ but not
$p^2$. Therefore, $P$ is isomorphic to the direct product of $m_d$
copies of $C_{p}$. Since $P$ is non-cyclic, $m_d>1$.

It is well-known that fusion of semisimple elements in a maximal
torus is controlled by its relative Weyl group (see \cite[Exercise
20.12]{malletesterman} or \cite[p. 6]{Malle-Maroti}). By choosing
$P$ to be inside the Sylow $d$-torus $\mathcal{S}_d$ and let
$\mathcal{T}_d$ be an $F$-stable maximal torus of $\GC$ containing
$\mathcal{S}_d$, we deduce that the fusion of $p$-elements in $P$ is
controlled by the relative Weyl group $W(\mathcal{T}_d)$ of
$\mathcal{T}_d$. Therefore, the number of conjugacy classes of
(non-trivial) $p$-elements of $\GG$, and hence of $S$, is at least
\[
\frac{|P|-1}{|W(\mathcal{T}_d)|}=\frac{p^{m_d}-1}{|W(\mathcal{T}_d)|}.
\]

Note that when $d$ is regular for $\GC$, which means that
$\bC_{\GC}(\mathcal{S}_d)$ is a maximal torus of $\GC$, the maximal
torus $\mathcal{T}_d$ can be chosen to be the same as
$\mathcal{L}_d=\bC_{\GC}(\mathcal{S}_d)$, and this indeed happens
for all exceptional types and all $d$, except the single case of
type $E_7$ and $d=4$ (see also \cite[p. 18]{Hung-ShaefferFry}).

Recall that $p\nmid |\bZ(\GG)|$, and thus $B_0(\GG)$ and $B_0(S)$
are isomorphic, and, moreover, $p$ is a good prime for $\GC$. By a
result of Geck~\cite[Theorem A]{Geck93}, the restrictions of
unipotent characters of $\GG$ in $B_0(\GG)$ to $p$-regular elements
form a basic set of Brauer characters of $B_0(\GG)$. In particular,
$l(B_0(S))=l(B_0(\GG))$ is precisely the number of unipotent
(ordinary) characters in $B_0(\GG)$, which in turns is the number
$k(W(\mathcal{L}_d))$ of irreducible characters of
$W(\mathcal{L}_d)$, as mentioned above.

\begin{proposition}\label{prop:exceptional}
Theorem \ref{thm:B0-Lie-type} holds for groups of exceptional Lie
types.
\end{proposition}

\begin{proof} We will keep the notation above. In particular, $\GG$ and
$\GG^\ast$ are finite reductive groups of respectively
simply-connected and adjoint type with $S=\GG/\bZ(\GG)\cong
[\GG^\ast,\GG^\ast]$. First we note that the Sylow $3$-subgroups of
simple groups of type $E_6$ or $\ta E_6$ are not abelian since their
Weyl group ($\SO(5,3)$) has a non-abelian Sylow $3$-subgroup. So we
have $p\nmid |\bZ(\GG)|$ in all cases.

We will follow the following strategy to prove the theorem for
exceptional types. Let $\GG_1$ be the extension of $\GG^\ast$ to
include field automorphisms. Let
$$H:=\langle G\cap \GG^\ast,\bC_{G\cap\GG_1}(P)\rangle.$$
Note that every unipotent character of
$S$ is $\GG_1$-invariant and extendible to its inertial subgroup in
$\Aut(S)$, by results of Lusztig and Malle (see \cite[Theorems 2.4
and 2.5]{Malle08}). In particular, every unipotent character in
$B_0(S)$ extends to a character in $B_0(\GG_1)$. The result of Geck
noted above then implies that each $\theta\in\IBr(B_0(S))$ extends
to some $\mu\in\IBr(B_0(G\cap \GG_1))$. 
Now $\mu_H\in\IBr(B_0(H))$. Moreover, as
$P\bC_{G\cap\GG_1}(P)\subseteq H$, $B_0(G\cap\GG_1)$ is the only
block of $G\cap \GG_1$ covering $B_0(H)$ (see \cite[Lemma
1.3]{Rizo}). It follows that $\mu \eta \in\IBR(B_0(G\cap\GG_1))$ for
every $\eta\in\IBR((G\cap \GG_1)/H)$ by \cite[Corollary 8.20 and
Theorem 9.2]{Navarro}, and thus
\begin{equation}\label{eq1}
l(B_0(G\cap\GG_1))\geq l(B_0(S))|(G\cap
\GG_1)/H|=k(W(\mathcal{L}_d))|(G\cap \GG_1)/H|.\end{equation}
Here
we remark that $(G\cap \GG_1)/H$ is a quotient of $(G\cap
\GG_1)/(G\cap \GG^\ast)$ and thus cyclic. Also, the number
$l(B_0(G))$ could be smaller than $l(B_0(G\cap\GG_1))$, depending on
how unipotent characters of $S$ are fused under graph automorphisms,
and this will be examined below in a case by case analysis.

Assume for now that $d$ is regular for $\GC$ (which means
$(\GC,d)\neq (E_7,4)$), we then choose
$\mathcal{T}_d:=\mathcal{L}_d$ as mentioned above. Recall that
$|P|=p^{m_d}$ and $S$ then has at least
$(p^{m_d}-1)/|W(\mathcal{L}_d)|$ conjugacy classes of non-trivial
$p$-elements. Assume that $G$ has a unique class of non-trivial
$p$-elements, and therefore we aim to prove that $l(B_0(G))\ge p$.
Since $\bC_G(P)$ fixes every class of $p$-elements of $S$, we deduce
that
\begin{equation}\label{eq2}
\frac{p^{m_d}-1}{|W(\mathcal{L}_d)|}\leq \frac{|G|}{|\langle
S,\bC_G(P)\rangle|}\leq d\frac{|G|}{|H|}\leq
dg\frac{|G\cap\GG_1|}{|H|},
\end{equation}
where $d$ and $g$ are respectively the orders of the groups of
diagonal and graph automorphisms of $S$.

We now go through various types of $S$ to reach the conclusion, with
the help of \eqref{eq1} and \eqref{eq2}. For simplicity, set
$x:=|(G\cap \GG_1)/H|$. The relative Weyl groups $W(\mathcal{L}_d)$
for various types of $\GC$ and $d$ are available in
\cite[Table~3]{Broue-Malle-Michel}. These relative Weyl groups are
always complex reflection groups and we will follow their notation
in \cite{Broue-Malle-Michel} as well as \cite{Benard76}. Recall that
as the Sylow $p$-subgroups of $S$ are non-cyclic, we may exclude the
types $\ta B_2$ and $\ta G_2$.

Let $S=G_2(q)$ with $q>2$. Then $d\in\{1,2\}$, $m_1=m_2=2$, and
$W(\mathcal{L}_d)$ is the dihedral group $D_{12}$. Here all
unipotent characters of $S$ are $\Aut(S)$-invariant unless $q=3^f$
for some odd $f$, in which case the graph automorphism fuses two
certain unipotent characters in the principal series, by a result of
Lusztig (see \cite[Theorem 2.5]{Malle08}). In any case, the bound
\eqref{eq1} yields $l(B_0(G))\geq (k(D_{12})-1)x=5x$. Together with
\eqref{eq2}, we have
\[
l(B_0(G))\geq5x>\sqrt{24x}\geq\sqrt{p^2-1}> p-1,
\]
as desired.

For $S=F_4(q)$ we have $d\in\{1,2,3,4,6\}$ with $m_1=m_2=4$ and
$m_3=m_4=m_6=2$. Here all unipotent characters of $S$ are
$\Aut(S)$-invariant unless $q=2^f$ for some odd $f$, in which case
the graph automorphism fuses eight pairs of certain unipotent
characters. Also, $W(\mathcal{L}_{1,2})=G_{28}$,
$W(\mathcal{L}_{3,6})=G_{5}$, and $W(\mathcal{L}_{4})=G_{8}$. In all
cases we have
\[l(B_0)\geq
(k(W(\mathcal{L}_d))-8)x>(2|W(\mathcal{L}_d)|x)^{1/m_d}\geq
(p^{m_d}-1)^{1/m_d}> p-1.\]

For all other exceptional types every unipotent character of $S$ is
$\Aut(S)$-invariant, again by \cite[Theorem 2.5]{Malle08}. The bound
\eqref{eq1} then implies that $l(B_0(G)\geq k(W(\mathcal{L}_d))x$.
On the other hand, the bound \eqref{eq2} yields
$dgx|W(\mathcal{L}_d)|\geq p^{m_d}-1$. The routine estimates are
then indeed sufficient to achieve the desired bound.

As the arguments for $\tb D_4$, $E_6$, $\ta E_6$, $E_7$ with $d\neq
4$, and $E_8$ are fairly similar, we provide details only for
$S=E_8(q)$ as an example. Then $d\in\{1,2,3,4,6,5,8,10,12\}$ with
$m_{1,2}=8$, $m_{3,4,6}=4$, and $m_{5,8,10,12}=2$. Going through
various values of $d$, we observe that
$k(W(\mathcal{L}_d))^{m_d}>|W(\mathcal{L}_d)|$ for all relevant $d$.
The above estimates then imply that
\[
l(B_0(G))^{m_d}\geq
k(W(\mathcal{L}_d))^{m_d}x>|W(\mathcal{L}_d)|x\geq p^{m_d}-1,
\]
which in turns implies that $l(B_0(G))\ge p$.

Assume that $S=E_7(q)$ and $d=4$. (Recall that $d=4$ is not regular
for type $E_7$.) Then $m_d=2$. By \cite[Table
1]{Broue-Malle-Michel}, $\mathcal{L}_d=\mathcal{S}_d.A_1^3$,
$W(\mathcal{L}_d)=G_8$ and $W(\mathcal{T}_d)$ is an extension of
$G_8$ by $C_2^3$ for any maximal torus $\mathcal{T}_d$ containing
$\mathcal{S}_d$. Note that here $\Out(S)$ is the direct product of
$C_{\gcd(2,q-1)}$ and $C_{f}$ where $q=\ell^f$ for some prime
$\ell$, and thus is abelian. Let $y:=|G/\langle S,\bC_G(P)\rangle|$
and arguing similarly as above, we have $l(B_0(G))\geq
k(W(\mathcal{L}_d))y=16y$ and $p^2-1\leq y|W(\mathcal{T}_d)|=768y$.
If $y\geq 3$ then $l(B_0(G))^2\geq 16^2y^2\geq 768y\geq p^2-1$, as
desired. If $y=1$ then $p\leq 23$, and since we are done if $p\leq
16$, we may assume that $p=17, 19$, or $23$, but for these primes,
$p^2-1$ does not divide $|W(\mathcal{T}_d)|=768$, implying that $S$,
and hence $G$, has more than one class of $p$-elements. Lastly, if
$y=2$ then the only prime we need to take care of is $p=37$, but as
$37^2-1=1368$ cannot be a sum of two divisors of
$|W(\mathcal{T}_d)|$, $S$ now has at least three classes of
$p$-elements, implying that $G$ has more than one class of
$p$-elements, as desired.

Finally, let $S=\ta F_4(q)$ with $q=2^{2n+1}\geq 8$. Here the prime
$p$ divides exactly one of $\Phi_1(q)$, $\Phi_2(q)$,
$\Phi_{4^+}(q)=q+\sqrt{2q}+1$, and $\Phi_{4^-}(q)=q-\sqrt{2q}+1$,
and $m_d=2$ in all cases. All the Sylow $d$-tori are maximal and
their relative Weyl groups are $D_{16}$ for $d=1$, $G_{12}$ for
$d=2$, and $G_8$ for $d=4^\pm$. Now one just applies \eqref{eq1} and
\eqref{eq2} to arrive at the desired bound.
\end{proof}

\begin{proposition}\label{prop:classical}
Theorem \ref{thm:B0-Lie-type} holds for groups of classical types.
\end{proposition}

\begin{proof}
First consider $S=\PSL^\epsilon(n,q)$ with $\epsilon=\pm$ and $n\geq
3$. Here, as usual, $\PSL^+(n,q):=\PSL(n,q)$ and
$\PSL^-(n,q):=\PSU(n,q)$. Let $e$ be the smallest positive integer
such that $p\mid (q^e-\epsilon^e)$.

Assume that $p\nmid |\bZ(\SL^\epsilon(n,q))|$, and thus we may view
$P$ as a (Sylow) $p$-subgroup of $\SL^\epsilon(n,q)$. Since $P$ is
not cyclic, we have $2e\leq n$. (If $2e>n$ then $P$ would be
contained in a torus of order $q^e-\epsilon^e$, and hence cyclic.)
Let $\alpha$ be an element of $\overline{\FF}^\times_q$ of order
$p$. We then can find an element $x_0\in \SL^\epsilon(e,q)$ of order
$p$ that is conjugate to $\diag(\alpha,\alpha^{\epsilon
q},\ldots,\alpha^{(\epsilon q)^{e-1}})$ over $\overline{\FF}_q$. Now
we observe that the two elements $x:=\diag(x_0,I_{n-e})$ and
$y:=\diag(x_0,x_0,I_{n-2e})$ of $\SL^\epsilon(n,q)$ produce two
corresponding elements of order $p$ in $S$ that cannot be conjugate
in $G$, as desired.

Now assume $p\mid |\bZ(\SL^\epsilon(n,q))|$. As $P\in\Syl_p(S)$ is
abelian, this happens only when $p=3$ (see \cite[Lemma
2.8]{KoshitanikB3}). The proof of \cite[Lemma
2.5]{Kulshammer-Navarro-Sambale-Tiep} shows that, in this case,
$S=\PSL^\epsilon(3,q)$ with $3\mid (q-\epsilon)$ but $9\nmid
(q-\epsilon)$. Moreover, $q=\ell^f$ for some prime $\ell$ with $3
\nmid f$, so the Sylow $3$-subgroups of $S$ (and $G$) are elementary
abelian of order 9. Suppose that $l(B_0(G)) \leq 2$. Then the
irreducible Brauer characters in $B_0(G)$ are $1_G$, and possibly
another character $\gamma$. On the other hand, it is known from
\cite[Theorem 4.5]{Geck90} and \cite[Table 1]{Kunugi00} that
$B_0(S)$ then contains precisely 5 distinct irreducible $3$-Brauer
characters, two of which, $1_S$ and $\alpha$, are linear
combinations of the restrictions of the two unipotent characters of
degrees $1$ and $q^2-\epsilon q$ to $3$-regular elements, and thus
are $G$-invariant; and three more $\beta_1, \beta_2,\beta_3$. It
follows that $\gamma$ lies above $\alpha$, but then none of the
Brauer characters of $B_0(G)$ can lie above $\beta_i$, a
contradiction. Hence $l(B_0(G)) \geq 3$, as required. (In fact,
Brou\'{e}'s abelian defect group conjecture, and hence the blockwise
Alperin weight conjecture, holds for principal $3$-blocks with
elementary abelian defect groups of order 9, see
\cite{Koshitani-Kunugi}, and thus the bound $l(B_0(G))\geq 3$ also
follows by Section~\ref{section-Alperin}.)

For symplectic and orthogonal types, note that as $p$ is odd, we may
view $P\in\Syl_p(S)$ as a Sylow $p$-subgroup of $\Sp$, $\SO$, and
$\GO$. Let $e$ be the smallest positive integer such that $p\mid
(q^{2e}-1)$. As above we have $2e\leq n$ by the non-cyclicity of
$P$.

Consider $S=\PSp(2n,q)$ with $n\geq 2$. Since
$\SL(2,q^e)<\Sp(2e,q)$, we may find an element $x_0$ in $\Sp(2e,q)$
of order $p$ with spectrum
$\{\alpha,\alpha^{q},\ldots,\alpha^{q^{e-1}},\alpha^{-1},\alpha^{q},\ldots,\alpha^{-q^{e-1}}\}$
(see the proof of \cite[Proposition 2.6]{Navarro-Tiep}). Note that
\[ \Sp(2e,q)\times \Sp(2e,q)\times \Sp(2n-4e,q)<\Sp(2e,q)\times
\Sp(2n-2e,q)<\Sp(2n,q).
\] Now one sees that the images of $x:=\diag(x_0,I_{2n-2e})$ and
$y:=\diag(x_0,x_0,I_{2n-4e})$ in $S$ are not conjugate in $G$.

Consider $S=\Omega(2n+1,q)$ with $q$ odd and $n\geq 3$. Since $p\mid
(q^{2e}-1)$, there is a (unique) $\lambda\in\{\pm1\}$ such that
$p\mid (q^e-\lambda)$. Using the embedding
\[C_{q^e-\lambda}\cong \SO^\lambda(2,q^e)<\GO^\lambda(2e,q),\] we may find $x_0\in
\GO^\lambda(2e,q)$ of order $p$ and with the spectrum
$\{\alpha^{\pm1},\alpha^{\pm q},\ldots,\alpha^{\pm q^{e-1}}\}$. This
$x_0$ then must be inside $\SO^\lambda(2e,q)$ since it has order
$p$. Note that
\[
\SO^\lambda(2e,q)\times \SO^\lambda(2e,q)\times
\SO(2n-4e+1,q)<\SO(2n+1,q).
\]
It follows that the images of $x:=\diag(x_0,I_{2n-2e+1})$ and
$y:=\diag(x_0,x_0,I_{2n-4e+1})$ in $S$ are of order $p$, and are
not conjugate in $G$.

For $S=\mathrm{P}\Omega^+(2n,q)$ with $n\geq 4$, using the same
element $x_0\in\SO^\lambda(2e,q)$ as in the case of odd-dimensional
orthogonal groups and the embedding
\[
\SO^\lambda(2e,q)\times \SO^\lambda(2e,q)\times
\SO^+(2n-4e,q)<\SO^+(2n,q),
\]
we arrive at the same conclusion.

Finally, consider $S=\mathrm{P}\Omega^-(2n,q)$ with $n\geq 4$. If
$n=2e$, then we have $p\mid (q^n-1)$ and it follows that the Sylow
$p$-subgroups of $S$ are in fact cyclic, which is not the case. So
$n\geq 2e+1$. As in the case of split orthogonal groups, but using
the embedding
\[
\SO^\lambda(2e,q)\times \SO^\lambda(2e,q)\times
\SO^-(2n-4e,q)<\SO^-(2n,q),
\]
we have that $G$ has at least two classes of non-trivial $p$-elements
as well. This finishes the proof.
\end{proof}

We have completed the proof of Theorem~\ref{thm:B0-Lie-type}, and
therefore the proof of Theorems~\ref{main-l(Bo)>p-1} and
\ref{theorem-kp'} as well.


\section{Groups with three \texorpdfstring{$p$}{p}-classes}\label{sec:kp(G)=3}

In this section we prove the following result, which provides a
bound for $k_{p'}(G)$ for groups $G$ with 3 conjugacy classes of
$p$-elements.

\begin{theorem}\label{main-kp(G)=3}
Let $G$ be a finite group with $k_p(G)=3$. Then $k_{p'}(G)\ge
(p-1)/2$ with equality if and only if $p>2$ and $G$ is the Frobenius
group $C_p\rtimes C_{(p-1)/2}$.
\end{theorem}

We will prove that Theorem~\ref{main-kp(G)=3} follows from
Theorem~\ref{psolv2}, \cite[Theorem 2.1]{Hung-Maroti} on bounding
the number of orbits of $p$-regular classes of simple groups under
their automorphism groups, the known cyclic Sylow case of the
blockwise Alperin weight Conjecture, and the following result.

\begin{theorem}\label{almost-simple-kp'>p}
Let $p$ be a prime and $S$ a finite simple group with non-cyclic
Sylow $p$-subgroups. Let $S\unlhd G \leq \Aut(S)$. Then
$k_{p'}(G)\geq p$.
\end{theorem}

\begin{proof} The theorem is clear when $p=2,3$ as $|G|$ has at least 3 prime
divisors. Therefore we may assume that $p\geq 5$. We also may assume
that $S$ is not a sporadic simple group or the Tits group, as these
could be checked directly using the character table library in
\cite{GAP48}.

Let $S=\Al_n$. Since the Sylow $p$-subgroups of $S$ are not cyclic,
we have $n\geq 2p\geq 10$. It follows that $\Al_n$ has at least
$p-1$ cycles of odd length not divisible by $p$. These cycles
together with an involution of $S$ produce at least $p$ $p$-regular
classes of $G$, as desired.

Next we assume that $S$ is a simple group of Lie type in
characteristic $p$. As before, one then can find a simple algebraic
group $\mathcal{G}$ of simply connected type defined in
characteristic $p$ and a Frobenius endomorphism $F$ such that
$S=\GG/\bZ(\GG)$, where $\GG={\mathcal{G}}^{F}$. According to
\cite[Theorem 3.7.6]{Carter}, the number of semisimple classes of
$\GG$ is $q^r$, where $q$ is the size of the underlying field of
$\mathcal{G}$ and $r$ is the rank of $\mathcal{G}$. Therefore,
\[k_{p'}(S)\geq
\frac{k_{p'}(\GG)}{k_{p'}(\bZ(\GG))}\geq \frac{q^r}{|\bZ(\GG)|}.\]
To prove the theorem in this case, it suffices to prove that
$q^r\geq p|\bZ(\GG)||\Out(S)|$. Using the known values of
$|\bZ(\GG)|$ and $|\Out(S)|$ available in \cite[p. xvi]{Atl1} for
instance, it is straightforward to check the inequality for all $S$
and relevant values of $q,r$ and $p$, unless $(S,p)$ is one of the
following pairs
\[
\{(\PSL(2,5^2),5),(\PSL(3,7),7),(\PSL(3,13),13),(\PSU(3,5),5),(\PSU(3,11),11)\}.
\]
Again the character tables of the corresponding almost simple groups
are available in \cite{GAP48} unless $S=\PSL(3,13)$. For this
exception we used the computer to find $13$ distinct pairs
$(|\langle x\rangle|,|\bC_S(x)|)$ where $x\in S$ is $p$-regular. Of
course these elements cannot be conjugate in $G$.

For the rest of the proof, we will assume that $S$ is a simple group
of Lie type in characteristic $\ell\neq p$ and let $\GG$ be a finite
reductive group of adjoint type with socle $S$. (Note that $\GG$
from now on is different from before where it denotes the finite
reductive group of simply-connected type.)

\begin{lemma}\label{lemma1} Let $S, G$ and $\GG$ as above. If $k_{p'}(\GG)\geq
p|\Out(S)|$, then $k_{p'}(G)\geq p$.
\end{lemma}

\begin{proof} Let $\Cl_{p'}(S)$ denote the set of $p$-regular classes of $S$ and
$n(H,\Cl_{p'}(S))$ the number of orbits of the action of a group $H$
on $\Cl_{p'}(S)$. Let $G_1:=\langle G\cup \GG\rangle$. Then
\begin{align*}
k_{p'}(G)&\geq
n(G,\Cl_{p'}(S))\geq n(G_1,\Cl_{p'}(S)) =\frac{1}{|G_1|}\Bigl(\sum_{c\in\Cl_{p'}(S)}|\Stab_{G_1}(c)|\Bigr)\\
&\geq
\frac{1}{|G_1|}\Bigl(\sum_{c\in\Cl_{p'}(S)}|\Stab_\GG(c)|\Bigr)=\frac{|\GG|}{|G_1|}n(\GG,\Cl_{p'}(S))\\
&\geq \frac{|\GG|}{|G_1|} \frac{k_{p'}(\GG)}{|\GG/S|}\geq
\frac{k_{p'}(\GG)}{|\Out(S)|}\geq p,
\end{align*}
as claimed.
\end{proof}

Recall that $p\geq 5$. As the Sylow $p$-subgroups of $S$, where $p$
is not the defining characteristic of $S$, are non-cyclic, $S$ is
not one of the types $A_1$, $\ta B_2$ and $\ta G_2$.

\smallskip

1. Let $\GG=\PGL^\epsilon(n,q)$ with $\epsilon=\pm$, $q=\ell^f$ and
$n\geq 3$. Here as usual we use $\epsilon=+$ for linear groups and
$\epsilon=-$ for unitary groups. Consider tori $T_i$
($i\in\{n-1,n\}$) of $\GG$ of size
$(q^{i}-(\epsilon1)^i)/(q-\epsilon1)$. Since
$\gcd(|T_{n-1}|,|T_n|)=1$, there exists $t\in\{n-1,n\}$ such that
$p\nmid |T_t|$. Note that the fusion of semisimple elements in $T_t$
is controlled by the relative Weyl group, say $W_t$, of $T_t$, which
is the cyclic group of order $t$ (see \cite[Proposition
5.5]{Malle-Maroti} and its proof for instance). Therefore, the
number of $p$-regular (semisimple) classes of $\GG$ with
representatives in $T_t$ is at least
\[
\frac{q^{t}-(\epsilon1)^i}{t(q-\epsilon1)}.
\]

Let $k\in\NN$ be the order of $q$ modulo $p$. Since the Sylow
$p$-subgroups of $S$ are not cyclic, we must have $n\geq 2k$. Now
one can check that
\[
\frac{q^{t}-(\epsilon1)^i}{t(q-\epsilon1)}\geq
2f\gcd(n,q-\epsilon1)p=|\Out(S)|p
\]
for all possible values of $q,n$ and $p$. It follows that
\[
k_{p'}(\GG)\geq |\Out(S)|p,
\]
and therefore we are done in this case by Lemma \ref{lemma1}.

\smallskip

2. Let $\GG=\SO(2n+1,q)$ or $\PCSp(2n,q)$ for $n\geq 2$ and
$q=\ell^f$. Since $p$ is odd, it does not divide both $q^n-1$ and
$q^n+1$. Let $T$ be a maximal torus of $G$ of order either $q^n-1$
or $q^n+1$ such that $p\nmid |T|$. The fusion of (semisimple)
elements in $T$ is controlled by its relative Weyl group, which is
cyclic of order $2n$ in this case. Therefore, the number of
conjugacy classes with representatives in $T$ is at least
$1+(q^n-2)/(2n)$, and it follows that
\[
k_{p'}(\GG)\geq 2+\frac{q^n-2}{2n},
\]
since $S$ has at least one non-trivial unipotent class.

Let $k\in\NN$ be minimal such that $p$ divides $q^{2k}-1$. Since the
Sylow $p$-subgroups of $S$ are non-cyclic, we must have $n\geq 2k$.
Let $n=2$. It then follows that $k=1$ and, as $p\geq 5$, we have
$q\geq 9$, and thus the desired inequality $2+(q^2-2)/4\geq
2fp=p|\Out(S)|$ follows easily. So let $n\geq 3$, and hence
$\Out(S)$ is cyclic of order $f\gcd(2,q-1)$. We now easily check
that
\[
2+\frac{q^n-2}{2n}\geq f\gcd(2,q-1)p
\]
for all the relevant values of $p,q$ and $n$, unless
$(n,p,q)=(4,5,2)$, and indeed in all cases we have
\[
k_{p'}(\GG)\geq p|\Out(S)|,
\]
and the theorem follows by Lemma \ref{lemma1}.

\smallskip

3. Let $\GG= \PCO^\epsilon(2n,q)$ with $\epsilon=\pm$, $q = \ell^f$
and $n\geq4$. Here $|\Out(S)|=2f\gcd(4,q^n-\epsilon1)$ unless
$(n,\epsilon)=(4,+)$, in which case
$|\Out(S)|=6f\gcd(4,q^n-\epsilon1)$. Similar to other classical
groups we have $n\geq 2k$ where $k$ is minimal such that $p\mid
(q^{2k}-1)$. First assume that $k\nmid n$. A maximal torus of $\GG$
of size $q^{n}-\epsilon1$ will then produce at least
\[1+\frac{q^{n}-\epsilon1}{2n}\] $p$-regular classes, which are sufficient
for the desired bound of $p|\Out(S)|$ unless
$(n,\epsilon,q,p)=(4,+,4,5)$, $(5,\pm,2,5)$. The bound still holds
for these exceptions since $S$ has elements of at least 5 different
orders coprime to $5$, which makes $k_{p'}(G)\geq 5=p$.

Now we may assume $k\mid n$. The case $\epsilon=-$ and $p\mid
(q^k-1)$ can be treated as above using a torus of size $q^n+1$, with
a note that $\gcd(q^n+1,q^k-1)\mid 2$ and thus every element in that
torus has order coprime to $p$. So we assume that $\epsilon=+$ or
$\epsilon=-$ and $p\mid (q^k+1)$.

Observe that now $\GG$ has tori of size $q^{n-1}\pm 1$ which
consists of $p$-regular elements. Also, the non-trivial conjugacy
classes with representatives in these two tori have only one
possible common class, which is an involution class. Therefore,
\[
k_{p'}(\GG)\geq
2+\frac{q^{n-1}-3}{2(n-1)}+\frac{q^{n-1}-1}{2(n-1)}=2+\frac{q^{n-1}-2}{n-1}:=h(q,n).
\]
It turns out that the desired bound $h(q,n)\geq p|\Out(S)|$ is
satisfied unless $(S,p)=(P\Omega_{12}^-(2),5)$,
$(P\Omega_{12}^+(2),7)$, $(P\Omega_{12}^+(3),13)$,
$(P\Omega_{16}^+(2),17)$, $(P\Omega_{16}^+(3),41)$,
$(P\Omega_{8}^+(4),5)$, or $(P\Omega_{8}^+(q),p)$ with $q\leq 29$
and $p\mid (q^2+1)$ but $p\nmid (q^2-1)$.

For the pairs $(S,p)=(P\Omega_{12}^-(2),5)$,
$(P\Omega_{12}^+(2),7)$, or $(P\Omega_{8}^+(4),5)$, one can confirm
the bound by just counting the prime divisors of $|S|$. For
$(S,p)=(P\Omega_{16}^+(2),17)$, by counting elements of certain
order in the two tori of sizes $2^7\pm 1$, we observe that $\GG=S$
has at least $126/14=9$ classes of $127$-elements, $42/14=3$ classes
of $43$-elements, and $84/14=6$ classes of 129-elements, which
implies that $G$ has at least 10 classes of $\{127,43,129\}$-elements,
and hence, by also including classes of elements of order 1,2, 3, 5,
7, 9, 13, we have the claimed bound. The same strategy also works
for $(P\Omega_{12}^+(3),13)$ and $(P\Omega_{16}^+(3),41)$.

We are now left with the case $(S,p)=(P\Omega_{8}^+(q),p)$ with
$q\leq 29$ and $p\mid (q^2+1)$ but $p\nmid (q^2-1)$. Using the lower
bound for the number of $\Aut(S)$-orbits on $p$-regular classes of
$S$ in the proof of \cite[Lemma 4.6]{Hung-Maroti}, we end up with
the open cases
\[(q,p)\in\{(4,17),(5,13),(8,13),(9,41),(11,61)\}.\]
These groups can be realized as permutation groups (of degree $21{,}435{,}888$ in the last case) in \cite{GAP48}. In each case we constructed enough random $p$-regular elements in $S$ and computed their centralizer orders to make sure that these elements are not conjugate in $G$.

\smallskip

4. Now we turn to the case where $S$ is of exceptional type
different from $\ta B_2$ and $\ta G_2$. To conveniently write the
order $|S|$ and its factors, we use $\Phi_d$ to denote the $d$th
cyclotomic polynomial over the rational numbers.

As above let $\GG$ be a finite reductive group (over a field of size
$q$) of adjoint type with socle $S$, and assume that
$|\GG|=q^N\prod_i \Phi_i(q)^{a(i)}$ for suitable positive integers
$a(i)$ and $N$. (Indeed, $N$ is the number of positive roots in the
root system corresponding to $S$.) First we assume that the Sylow
$p$-subgroups of $\GG$ are not abelian. Then $p$ must divide the
order of the Weyl group of $\GG$, and thus $\GG$ is one of the types
$E_6$, $\ta E_6$, $E_7$, and $E_8$ and $p\leq 7$. Using elementary
number theory, one now easily observes that $|S|$ has at least $8$
different prime divisors, and hence $k_{p'}(G)\geq 7\geq p$. So we
may and will assume that the Sylow $p$-subgroups of $\GG$ (and $S$)
are abelian (but not cyclic). It follows that there exists a unique
$d\in \NN$ such that $p\mid \Phi_d(q)$ and $a(d)\geq 2$, see
\cite[Lemma 25.14]{malletesterman}.

Let $\GG=G_2(q)$ with $q=\ell^f\geq 3$. We then have $p\mid
(q^2-1)$. Note that $\GG=S$ has maximal tori of coprime orders
$\Phi_3(q)$ and $\Phi_6(q)$, which are furthermore coprime to $p$
since $p\geq 5$. The relative Weyl groups of these tori have order 6
(see \cite[Tables 1 and 3]{Broue-Malle-Michel} for sizes of Weyl
groups of various maximal tori). Therefore,
\[
k_{p'}(\GG)\geq
1+\frac{\Phi_3(q)-1}{6}+\frac{\Phi_6(q)-1}{6}=1+\frac{q^2}{3}.
\]
It is now sufficient to check that $1+q^2/3\geq fp$, but this is
straightforward. The case $S=F_4(q)$ is handled similarly by
considering two maximal tori of orders $\Phi_8(q)$ and
$\Phi_{12}(q)$.

Let $S=\ta F_4(q)$ with $q=2^{2m+1}\geq 8$. Then we have $p\mid
\Phi_1(q)\Phi_2(q)\Phi_4(q)$. Using maximal tori of orders
$\Phi_{12}^{\pm}(q):=q^2\pm \sqrt{2q^3}+q\pm \sqrt{2q}+1$ with the
relative Weyl group of order $12$, we end up with
\[
k_{p'}(G)\geq
1+\frac{\Phi_{12}^+(q)-1}{12}+\frac{\Phi_{12}^-(q)-1}{12}=1+\frac{q^2+q}{6}.
\]
Certainly $1+(q^2+q)/6\geq (2m+1)p$ for all possible values of $p$
and $m$ except $(p,m)=(13,1)$, but this exception can be checked
directly using \cite{GAP48}.

Let $S=\tb D_4(q)$. We then have $p\mid
\Phi_1(q)\Phi_2(q)\Phi_3(q)\Phi_6(q)$. The relative Weyl group of a
maximal torus of order $\Phi_{12}(q)$ has order 4, and thus the
number of classes with representatives in this torus is at least
$1+(q^4-q^2)/4$, which in turn is at least $3fp$, as we wanted,
unless $(q,p)=(2,7)$, $(3,13)$, or $(4,13)$. The first exception can
be handled easily using \cite{Atl1}. Let $(S,p)=(\tb D_4(4),13)$. We
already know that $S$ has at least $(4^4-4^2)/4=60$ classes of
elements of order $\Phi_{12}(4)=241$ and thus, as $\Out(S)$ is
cyclic of order 6, we are done unless $G=\Aut(S)$. In fact, even for
$G=\Aut(S)$, one just notices that $G$ has at least $60/6=10$
classes of elements of order 241, and therefore, together with
classes of elements of order 1,2 and 3, the desired bound follows.
Finally let $(S,p)=(\tb D_4(3),13)$. Then $S$ has at least $18$
classes of elements of order $\Phi_{12}(3)=73$, which produces at
least $6$ classes for $G$. Now note that $\SL(2,27)\leq S$, and by
using \cite{Atl1}, we then observe that $\SL(2,27)$, and hence $S$,
has elements of orders 1, 2, 3, 4, 6, 7, 14, which produce 7 more
13-regular classes, as wanted.

For $\GG=E_6(q)_{ad}$, we have $p\mid \Phi_d(q)$ for some
$d\in\{1,2,3,4,6\}$. Consider the (semisimple) classes with
representatives in a maximal torus of size $\Phi_{9}(q)$, with
notice that this torus has the relative Weyl group of order 9, we
obtain
\[
k_{p'}(\GG)\geq 1+\frac{q^6+q^3}{9},
\]
which is certainly at least $|\Out(S)|p$ for every relevant $q$ and
$p$. Similar arguments also work for $\GG=\ta E_6(q)_{ad}$ and
$E_7(q)_{ad}$, but using a maximal torus of respectively size
$\Phi_{18}(q)$ and $\Phi_1(q)\Phi_9(q)$ or $\Phi_2(q)\Phi_9(q)$,
depending on which size is coprime to $p$. For $\GG=E_8(q)$ with
$q=\ell^f$ we have $p\mid \Phi_d(q)$ for some
$d\in\{1,2,3,4,5,6,8,10,12\}$ and $\GG$ has a maximal torus of size
$\Phi_{30}(q)$ with the relative Weyl group of order 30, and it
follows that
\[
k_{p'}(\GG)\geq 1+\frac{\Phi_{30}(q)-1}{30}\geq fp,
\]
as desired. This concludes the proof of
Theorem~\ref{almost-simple-kp'>p}.
\end{proof}


We are now in the position to prove the main result of this section.

\begin{proof}[Proof of Theorem~\ref{main-kp(G)=3}]
Assume that the theorem is false and let $G$ be a minimal
counterexample. In particular, $G$ is not isomorphic to the
Frobenius group $F_p:=C_p\rtimes C_{(p-1)/2}$ and $k_{p'}(G)\leq
(p-1)/2$. Since $k_p(G/\mathbf{O}_{p'}(G))=k_p(G)=3$ and
$k_{p'}(G/\mathbf{O}_{p'}(G))\leq k_{p'}(G)$, we have
$\mathbf{O}_{p'}(G)=1$ or $G/\mathbf{O}_{p'}(G)\cong F_p$. In the
latter case $G$ has a cyclic Sylow $p$-subgroup and hence cannot be
a counterexample, as shown in Section~\ref{section-Alperin}, and
thus we have $\mathbf{O}_{p'}(G)=1$. Let $N$ be a minimal normal
subgroup of $G$. It follows that $p\mid |N|$, and hence $k_p(G/N)<
k_p(G)$. Now since $k_p(G/N)$ cannot be 2 by
Theorem~\ref{main-l(Bo)>p-1}, we must have $p\nmid |G/N|$ and
moreover, $N$ is the unique minimal normal subgroup of $G$.

We are done if $N$ is abelian by Theorem~\ref{psolv2}. So we may
assume that $N$ is a direct product of, say $n$, copies of a
non-abelian simple group, say $S$. Note that $p\mid |S|$ since
$p\mid |N|$. Therefore, the assumption $k_p(G)=3$ implies that
$n\leq 2$.

Assume that $n=2$. Let $m(S,p)$ be the number of $\Aut(S)$-orbits on
$p$-regular classes of $S$. We then have
\[
k_{p'}(G)\geq\frac{1}{2}m(S,p)(m(S,p)+1).
\]
It was shown in \cite[Theorem 2.1]{Hung-Maroti} that either
$m(S,p)>2\sqrt{p-1}$ or $(S,p)$ belongs to a list of possible
exceptions described in \cite[Table 1]{Hung-Maroti}. For the former
case, we have
\[
k_{p'}(G)> \frac{2\sqrt{p-1}(2\sqrt{p-1}+1)}{2}> \frac{p-1}{2},
\]
which is a contradiction. For the latter case, going through the
list of exceptions, we in fact still have
\[\frac{m(S,p)(m(S,p)+1)}{2}>\frac{p-1}{2},\] which again leads to a contradiction.

Finally we may assume that $n=1$, which means that $G$ is an almost
simple group with socle $S$. Furthermore, $p\nmid |G/S|$. The
theorem now follows from Section~\ref{section-Alperin} when Sylow
$p$-subgroups of $S$ are cyclic and from
Theorem~\ref{almost-simple-kp'>p} otherwise. This completes the
proof.
\end{proof}


\section{Theorem \ref{main-ko>p} and further applications}\label{section-applications}

We now derive Theorem \ref{main-ko>p}, which is restated, from
Theorem~\ref{main-l(Bo)>p-1}.

\begin{theorem}
Let $p$ be a prime and $G$ a finite group in which all non-trivial
$p$-elements are conjugate. Let $B_0$ denote the principal $p$-block
of $G$ Then $k_{0}(B_0)\geq p$ or $p = 11$ and $k_{0}(B_0)=10$.
\end{theorem}

\begin{proof} The theorem follows from Theorem~\ref{main-l(Bo)>p-1}
and \cite{Kessar-Malle} when the Sylow $p$-subgroups of $G$ are
abelian. Assume otherwise. Then, as mentioned before, by
\cite[Theorem 1.1]{Kulshammer-Navarro-Sambale-Tiep}, either

\begin{enumerate}[\rm(a)]
\item $p = 3$ and $\bO^{p'}(G/\bO_{p'}(G))$ is isomorphic to $Ru$,
$J_4$ or $\ta F_4(q)'$ with $q = 2^{6b\pm1}$ for a nonnegative
integer $b$, or

\item $p = 5$ and $G/\bO_{p'}(G)$ is isomorphic to $Th$.
\end{enumerate}
We now just proceed as in the proof of
Theorem~\ref{main-l(Bo)>p-1repeat}, but with height $0$ characters
instead of Brauer characters. For (b) we have
$k_0(B_0)=k_0(B_0(Th))=20>5$, and we are done. For (a) we may assume
that $G$ is almost simple. As
\[k_0(B_0(Ru))=k_0(B_0(J_4))=k_0(B_0(\ta F_4(2)'))=k_0(B_0(\ta F_4(2)))=9\]
by \cite{GAP48}, we may now assume that $S=\ta F_4(q)'$ with $q =
2^{6b\pm1}$ for some $b\in\ZZ^+$ and $S\unlhd G \leq \Aut(S)$.
According to \cite[\S 6 and \S 7]{Malle90}, the principal 3-block of
$S$ contains the Steinberg character denoted by $\chi_{21}$ (of
degree $q$), the semisimple character denoted by $\chi_{5,1}$ (of
degree $(q-1)(q^2+1)^2(q^4-q^2+1)$), and the trivial character, all
of which are $3'$-degree and $\Aut(S)$-invariant, implying that
$k_0(B_0)\geq 3$. The theorem is fully proved.
\end{proof}

Finally, we provide some more examples of applications of
Theorem~\ref{main-l(Bo)>p-1} in the study of principal blocks with
few characters.

\begin{theorem}\label{thm:app1}
Let $G$ be a finite group with a Sylow $p$-subgroup $P$ and the
principal $p$-block $B_0$. Assume that $k(B_0)=5$ and $l(B_0)=4$.
Then $P\cong C_5$.
\end{theorem}

\begin{proof}
By Theorem \ref{main-l(Bo)>p-1}, we have $p\leq 5$. By \cite[Theorem
3.6]{Kulshammer-Navarro-Sambale-Tiep}, $P$ is (elementary) abelian.
It then follows by \cite{Kessar-Malle} that the ordinary irreducible
characters in $B_0$ all have $p'$-degree, and thus $k_{0}(B_0)=5$.
However, by \cite[Corollaries 1.3 and 1.6]{Landrock81},
$5=k_{p'}(B_0)$ is divisible by $p$ if $p=2$ or 3, which cannot
happen.

So we are left with $p=5$. The equality part of Theorem
\ref{main-l(Bo)>p-1} then implies that
$\bN_G(P)/\mathbf{O}_{p'}(\bN_G(P))$ is isomorphic to the Frobenius
group $C_p\rtimes C_{{p-1}}$. In particular, $P\cong C_5$, as
wanted.
\end{proof}

\begin{theorem}\label{thm:app2}
Let $G$ be a finite group with a Sylow $p$-subgroup $P$ and the
principal $p$-block $B_0$. Assume that $k(B_0)=l(B_0)+1=7$. Then
$P\cong C_7$.
\end{theorem}

\begin{proof}
Again by Theorem \ref{main-l(Bo)>p-1}, we have $p\leq 7$ and as
above, the cases $p=2$ or 3 do not occur by \cite[Corollaries 1.3
and 1.6]{Landrock81}. If $p=7$ then the equality part of
Theorem~\ref{main-l(Bo)>p-1} implies that
$\bN_G(P)/\mathbf{O}_{p'}(\bN_G(P))\cong C_7\rtimes C_6$, yielding
that $P\cong C_7$, as claimed.

We now eliminate the possibility $p=5$. Assume so. By
Theorem~\ref{thm:k-l=1}, $G$ is not $p$-solvable and has a
non-cyclic Sylow $p$-subgroup. As in the proof of
Theorem~\ref{main-l(Bo)>p-1repeat}, we may assume that $G$ is an
almost simple group with a socle $S$ of Lie type in characteristic
not equal to $p$ and $p\nmid |G/S|$. Moreover, $P$ is abelian but
non-cylic. The proof of Proposition~\ref{prop:classical} then shows
that, when $S$ is of classical type, $G$ has more than one class of
non-trivial $p$-elements, contradicting the assumption that
$k(B_0)-l(B_0)=1$. Also, the proof of
Proposition~\ref{prop:exceptional} shows that $l(B_0)\geq 7$ when
$S$ is of exceptional types except possibly type $G_2$. (Indeed, the
principal block of $G_2(q)$ has exactly 6 irreducible modular
characters when $p\mid \Phi_{1,2}(q)=q\pm 1$, since
$k(W(\mathcal{L}_{1,2}))=k(D_{12})=6$.) So assume $S=G_2(q)$. Note
that, since $P$ is not cyclic, $5=p\mid (q\pm 1)$ and hence $q$ is
not an odd power of 3, implying that every unipotent character of
$S$ (including 6 in $B_0(S)$) is $\Aut(S)$-invariant, by
\cite[Theorem 2.5]{Malle08}. However, a quick inspection of the
principal block of $G_2(q)$ (see \cite[Theorems A and B]{Hiss89})
reveals that it contains two (families of) non-unipotent characters
of different degrees, implying that $B_0(G)$ contains at least 2
irreducible ordinary characters lying over non-unipotent characters
of $S$. It follows that $k(B_0(G))\geq 6+2=8$. This final
contradiction completes the proof.
\end{proof}

We conclude by noting that, while Theorem~\ref{thm:app1} can also be deduced
from the main result of \cite{Rizo} on principal blocks with exactly
5 irreducible ordinary characters, Theorem~\ref{thm:app2} is new.

\section*{Acknowledgments}

The authors are grateful to Gunter Malle for the careful reading of
an earlier version of the paper and providing helpful suggestions
that led to a clearer exposition. We also thank Attila Mar\'{o}ti for several fruitful
discussions on $d$-cuspidal pairs and their relative Weyl groups,
and $p$-regular classes.



\begin{thebibliography}{ABCD}

\bibitem[Alp]{Alperin}
J.\,L. Alperin, Weights for finite groups, \emph{Proc. Symp. Pure
Math.} \textbf{47} (1987), 369--379.

\bibitem[Ben]{Benard76}
M. Benard, Schur indices and splitting fields of the unitary
reflection groups, \emph{J. Algebra} \textbf{38} (1976), 318--342.

\bibitem[BMM]{Broue-Malle-Michel}
M. Brou\'{e}, G. Malle, J. Michel, Generic blocks of finite
reductive groups, \emph{Ast\'{e}risque} \textbf{212} (1993), 7--92.

\bibitem[CE]{Cabanes-Enguehard94}
M. Cabanes and M. Enguehard, On unipotent blocks and their ordinary
characters, \emph{Invent. Math.} \textbf{117} (1994), 149--164.

\bibitem[Car]{Carter}
  R.\,W. Carter, {\it Finite groups of Lie type. Conjugacy classes and complex
  characters}, Wiley and Sons, New York et al, 1985.

\bibitem[Atl]{Atl1}
  J.\,H. Conway, R.\,T. Curtis, S.\,P. Norton, R.\,A. Parker, and R.\,A. Wilson,
 {\it Atlas of finite groups}, Clarendon Press, Oxford, 1985.

\bibitem[Dad]{Dade66}
E.\,C. Dade, Blocks with cyclic defect groups, \emph{Ann. of Math.}
\textbf{84} (1966), 20--48.

\bibitem[DM]{Digne-Michel91}
F. Digne and J. Michel, \emph{Representations of finite groups of
Lie type}, London Mathematical Society Student Texts \textbf{21},
1991, 159 pp.

\bibitem[DM]{DM}
J.~D. Dixon and B. Mortimer, \textit{Permutation groups}, Graduate Texts in
  Mathematics, Vol. 163, Springer-Verlag, New York, 1996.

\bibitem[GAP]{GAP48}
The GAP~Group, \textit{GAP -- Groups, Algorithms, and Programming,
Version 4.11.0}, 2020. \url{http://www.gap-system.org}

\bibitem[Gec1]{Geck90}
M. Geck, Irreducible Brauer characters of the $3$-dimensional
special unitary groups in non-defining characteristic, \emph{Comm.
Algebra}, \textbf{18} (2000), 563--584.

\bibitem[Gec2]{Geck93}
M. Geck, Basic sets of Brauer characters of finite groups of Lie
type II, \emph{J. London Math. Soc.} \textbf{47} (1993), 255--268.

\bibitem[HKKS]{HKKS}
L. H\'{e}thelyi, R. Kessar, B. K\"{u}lshammer and B. Sambale, Blocks
with transitive fusion systems, \emph{J. Algebra} \textbf{424}
(2015), 190--207.

\bibitem[Him]{Himstedt11}
F. Himstedt, On the decomposition numbers of the Ree groups $\ta
F_4(q^2)$ in non-defining characteristic, \emph{J. Algebra}
\textbf{325} (2011), 364--403.

\bibitem[His]{Hiss89}
G. Hiss, On the decomposition numbers of $G_2(q)$, \emph{J. Algebra}
\textbf{120} (1989), 339--360.

\bibitem[HM]{Hung-Maroti}
N.\,N. Hung and A. Mar\'{o}ti, $p$-Regular conjugacy classes and
$p$-rational irreducible characters, \emph{J. Algebra, the special
issue dedicated to Jan Saxl}, to appear, 2020.
\href{https://arxiv.org/abs/2004.05194}{arXiv:2004.05194}

\bibitem[HSF]{Hung-ShaefferFry}
N.\,N. Hung and A.\,A. Shaeffer Fry, On
H\'{e}thelyi-K\"{u}lshammer's conjecture for principal blocks,
preprint, 2021.


\bibitem[Hup]{Huppert}
B. Huppert, \textit{Endliche {G}ruppen. {I}}, Grundlehren der
Mathematischen Wissenschaften, {\bf 134}, Springer-Verlag, Berlin,
1967.

\bibitem[Isa]{Isaac}
I.\,M. Isaacs, \emph{Finite Group Theory}, Graduate studies in
Mathematics \textbf{92}, American Mathematical Society, Providence,
Rhode Island, 2008.

\bibitem[KM1]{Kessar-Malle}
R. Kessar and G. Malle, Quasi-isolated blocks and Brauer's height
zero conjecture, \emph{Ann. of Math.} \textbf{178} (2013),
321--384.

\bibitem[KM2]{kessar-Malle2}
R. Kessar and G. Malle, Lusztig induction and $l$-blocks of finite
reductive groups, \emph{Pacific J. Math.} \textbf{279} (2015),
267--296.

\bibitem[KK]{Koshitani-Kunugi}
S. Koshitani and N. Kunugi, Brou\'{e}'s conjecture holds for
principal $3$-blocks with elementary abelian defect group of order
9, \emph{J. Algebra} {\textbf{248}} (2002), 575--604.

\bibitem[KS]{KoshitanikB3}
S. Koshitani and T. Sakurai, The principal $p$-blocks with small
numbers of characters, 2020.
  \href{https://arxiv.org/abs/2001.09970v2}{arXiv:2001.09970v2}.

\bibitem[KNST]{Kulshammer-Navarro-Sambale-Tiep}
B. K\"{u}lshammer, G. Navarro, B. Sambale and P.\,H. Tiep, Finite
groups with two conjugacy classes of $p$-elements and related
questions for $p$-blocks, \emph{Bull. London Math. Soc.} \textbf{46}
(2014), 305--314.

\bibitem[Kun]{Kunugi00}
N. Kunugi, Morita equivalent $3$-blocks of the $3$-dimensional
projective special linear groups, \emph{Proc. London Math. Soc.}
\textbf{80} (2000), 575--589.

\bibitem[Lan]{Landrock81}
P. Landrock, On the number of irreducible characters in a 2-block,
\emph{J. Algebra} \textbf{68} (1981), 426--442.

\bibitem[Lie]{Liebeckrk3}
M.\,W. Liebeck, The affine permutation groups of rank three,
\emph{Proc. London Math. Soc.} \textbf{54} (1987), 477--516.

\bibitem[Mal1]{Malle90}
 G. Malle, Die unipotenten Charaktere von $\ta F_4(q^2)$, \emph{Comm. Alg.} \textbf{18} (1990),
 2361--2381.

\bibitem[Mal2]{Malle08}
G. Malle, Extensions of unipotent characters and the inductive McKay
condition, \emph{J. Algebra} \textbf{320} (2008), 2963--2980.



\bibitem[Mal3]{Malle14}
G. Malle, On the inductive Alperin-McKay and Alperin weight
conjecture for groups with abelian Sylow subgroups, \emph{J.
Algebra} \textbf{397} (2014), 190--208.

\bibitem[MM]{Malle-Maroti}
G. Malle and A. Mar\'{o}ti, On the number of $p'$-degree characters
in a finite group, \emph{Int. Math. Res. Not.} \textbf{20} (2016),
6118--6132.

\bibitem[MR]{malle-robinson}
G. Malle and G.\,R. Robinson, On the number of simple modules in a
block of a finite group, \emph{J. Algebra} \textbf{475} (2017),
423--438.

\bibitem[MT]{malletesterman}
G.~Malle and D.~Testerman, \emph{Linear algebraic groups and finite
 groups of Lie type},
 Cambridge University Press, Cambridge, 2011.



\bibitem[Nav]{Navarro}
G. Navarro, \textit{Characters and blocks of finite groups}, London
  Mathematical Society Lecture Note Series, {\bf 250}, Cambridge University
  Press, Cambridge, 1998.

\bibitem[NT]{Navarro-Tiep}
G. Navarro and P.\,H. Tiep, Abelian Sylow subgroups in a finite
group, \emph{J. Algebra} \textbf{398} (2014), 519--526.

\bibitem[Pas]{Passman}
D.\,S. Passman, {$p$}-Solvable doubly transitive permutation groups,
  \emph{Pacific J. Math} \textbf{26} (1968), 555--577.

\bibitem[RSV]{Rizo}
N. Rizo, A.\,A. Schaeffer Fry, and C. Vallejo, Principal blocks with
5 irreducible characters, 2020.
\href{https://arxiv.org/abs/2010.15422}{arXiv:2010.15422}

\bibitem[Sam1]{habil}
B. Sambale, \textit{Blocks of finite groups and their invariants},
Springer Lecture Notes in Math., {\bf 2127}, Springer-Verlag, Cham,
2014.

\bibitem[Sam2]{SambaleBroue}
B. Sambale, Brou\'{e}'s isotypy conjecture for the sporadic groups
and their covers and automorphism groups, \emph{Internat. J. Algebra
Comput.} \textbf{25} (2015), 951--976.

\bibitem[SW]{Sawabe-Watanabe}
M. Sawabe and A. Watanabe, On the principal blocks of finite groups
with abelian Sylow $p$-subgroups, \emph{J. Algebra} \textbf{237}
(2001), 719--734.

\bibitem[Zas]{Zassenhaus1}
H. Zassenhaus, {\"{U}ber endliche {F}astk\"orper}, \emph{Abh. Math.
Sem. Univ.  Hamburg} \textbf{11} (1935), 187--220.

\end{thebibliography}
\end{document}